\documentclass[a4paper, 11pt]{amsart}
\makeatletter
\@namedef{subjclassname@2020}{%
	\textup{2020} Mathematics Subject Classification}
\makeatother

\usepackage{amsmath,amsthm,verbatim,amssymb,amsfonts,mathrsfs,amscd, graphicx,enumerate,tikz-cd,stmaryrd}
\usepackage{enumitem}
\setlist[enumerate,1]{label={$(\roman*)$},leftmargin=*}
\usepackage{multicol, caption}
\usepackage{fullpage}

\usepackage{mathtools}
\usepackage{xifthen}
\usepackage{xstring}
\usepackage{mathabx}
\usepackage{array}  
\usepackage{parskip} 
\usepackage{booktabs}
\usepackage[all]{xy}
\usepackage{hyperref} 
\hypersetup{
    colorlinks=true,
    linkcolor=blue,
    filecolor=magenta,      
    urlcolor=cyan,
    citecolor=blue,
}

\usepackage{booktabs}
\usepackage[OT2,T1]{fontenc}

\usepackage{cleveref}
\usepackage{umoline}

\usepackage{shortcuts}
\usepackage{stackengine}

\usepackage[msc-links,alphabetic]{amsrefs}

\newtheorem{thm}{Theorem}[section]
\newtheorem{lem}[thm]{Lemma}
\newtheorem{prop}[thm]{Proposition}
\theoremstyle{definition}
\newtheorem{defn}[thm]{Definition}

\newtheorem{exmp}[thm]{Example}
\newtheorem{cor}[thm]{Corollary}
\newtheorem{remark}[thm]{Remark}

\newtheorem*{question*}{Question}

\numberwithin{thm}{section}
\numberwithin{equation}{section}

\newcommand{\dotifempty} [1]{\ifthenelse{\isempty{#1}}
                          	{\cdot}%
                          	{#1}}

\DeclarePairedDelimiter\absolute{\lvert}{\rvert}%

\DeclarePairedDelimiter\braces{\{}{\}}

\newcommand{\sog}[1]{\left( #1 \right)}

\newcommand{\abs}[1]{\absolute{\dotifempty{#1}}}

\graphicspath{ {images/} }

\title{On the relation between pseudocharacters and Chenevier's determinants}

\author{Amit Ophir}

\begin{document}

\maketitle

\begin{abstract}
Consider a commutative unital ring $A$ and a unital $A$-algebra $R$. 
Let $d$ be a positive integer.
Chenevier proved that when $(2d)!$ is invertible in $A$, the map associating to a determinant its trace is a bijection between $A$-valued $d$-dimensional determinants of $R$ and $A$-valued $d$-dimensional pseudocharacters of $R$.
In this paper, we show that assuming $d!$ is invertible in $A$ is sufficient. This assumption is already made in the definition of a $d$-dimensional pseudocharacter.

Our proof involves establishing a product formula for pseudocharacters, which might be of independent interest.
\end{abstract}

\section{Introduction}

Pseudocharacters were introduced by Wiles (\cite{Wiles_OnOrdinary}) in dimension $2$ and by Taylor (\cite{Taylor}) more generally (under the name pseudo-representations) in order to construct $l$-adic representations of Galois groups.
They can be thought of as an abstractization of the trace of a finite dimensional representation.
However, the theory of pseudocharacters works best only under some restrictions on the characterstic of the ground ring.
This issue was solved by the introduction of \textit{determinants} by Chenevier in \cite{chenevier2014p}.
Determinant work over an arbitrary ring, with no assumption on its characteristic.
They are closely related to pseudocharacters, but are able to capture the notion of the characteristic polynomial of a finite dimensional representation.
Since their introduction, determinants have become a ubiquitous tool in number theory (\cite{scholze2015torsion} being one example).

In this short paper, we address the question: Under what conditions on the ground ring are pseudocharacters equivalent to determinants?

Let $A$ be a commutative unital ring, and let $R$ be a (not necessarily commutative) unital $A$-algebra.
To any $d$-dimensional $A$-valued determinant $D$ of $R$, there is an associated trace $\mathrm{Tr}_D:R\map A$, as defined in \cite{chenevier2014p} by 
\[D(t-x)=t^n-\mathrm{Tr}_D(x)t^{n-1}+\dots.\]
Chenevier showed (loc.\ Cit.\  Lemma 1.12, (iii)) that if $d!$ is invertible in $A$, then $\mathrm{Tr}_D$ is a $d$-dimensional pseudocharacter.
The mapping $D\mapsto \mathrm{Tr}_D$, from the set of $d$-dimensional $A$-valued determinants into the set of $d$-dimensional $A$-valued pseudocharacters, is shown to be injective (loc.\ Cit.\ Proposition 1.27).
Furthermore, Chenevier established (loc.\ Cit.\ Proposition 1.29) that this mapping becomes a bijection if $(2d)!$ is invertible in $A$.
However, he suggested (loc.\ Cit.\ Remark 1.28) that this might hold under the weaker condition that $d!$ is invertible in $A$.
Our main result verifies this:

\begin{thm}\label{thm_main}
Assume that $d!$ is invertible in $A$.
The mapping $D\mapsto Tr_D$ is a bijection between $A$-valued $d$-dimensional determinants on $R$ and $A$-valued $d$-dimensional pseudocharacters on $R$.
\end{thm}

A function $f:R\map A$ is called \textit{central} if $f(xy)=f(yx)$, for all $x,y\in R$.
An $A$-valued $d$-dimensional \textit{pseudocharacter} of $R$ is an $A$-linear central function $f:R\map A$ satisfying the following conditions.
    \begin{enumerate}
    \item $f(1)=d$,
    \item $d!$ is invertible in $A$, and
    \item $\pse{f}{d+1}(f)(x_1,\dots,x_{d+1})=0$ for all $x_1,\dots,x_{d+1}\in R$,
    \end{enumerate}
where the maps $\pse{f}{n}$
\footnote{Let $x_1,x_2,x_3\in R$. The first three $\pse{f}{n}$ are
        \begin{align*}
        \pse{f}{1}(x_1)&=f(x_1),\\
        \pse{f}{2}(x_1,x_2)&=f(x_1)f(x_2)-f(x_1x_2), \text{ and}\\
        \pse{f}{3}(x_1,x_2,x_3)&=f(x_1)f(x_2)f(x_3)-f(x_1x_2)f(x_3)-f(x_1x_3)f(x_2)-f(x_2x_3)f(x_1)+f(x_1x_2x_3)+f(x_1x_3x_2).
        \end{align*}}
        , for $n=1,2,\dots$, are defined recursively as follows.
Set $\pse{f}{1}(x)=f(x)$, and for any $n>1$,
    \begin{equation}\label{eq_recursive}
    \begin{aligned}
    \pseudo_nf(x_1,\dots,x_n)=&f(x_n)\pseudo_{n-1}f(x_1,\dots,x_{n-1})\\
    &-\sum_{i=1}^{n-1}\pseudo_{n-1}f(x_1,\dots,x_{i-1},x_i\cdot x_n,x_{i+1},\dots,x_{n-1}).
    \end{aligned}
    \end{equation} 

    \begin{remark}
    The definition of the maps $\pse{f}{n}$ given here is different than the definition given in \cite{Taylor}.
    See \cite{rouquier1996}*{Lemma 2.2.} for the equivalence of the two definitions.
    \end{remark} 

Given an $A$-linear representation $\rho:R\map \End_A(A^d)$, and if $d!$ is invertible in $A$, the map $\mathrm{Tr}\circ \rho:R\map A$ is an $A$-valued pseudocharacter of dimension $d$, as follows from an old result of Frobenius (see also \cite{Bellaiche_Chenevier_book}*{1.2.2. Main example}). 
Many properties of pseudocharacters are now known due to the work of Taylor (\cite{Taylor}), Rouquier (\cite{rouquier1996}), Nyssen (\cite{nyssen1996pseudo}), Proccesi (\cite{PROCESI1976}), and others

A $d$-dimensional $A$-valued determinant is an $A$-valued, multiplicative polynomial law, which is homogeneous of degree $d$.
We refer to Section $1$ of \cite{chenevier2014p} for the definitions of these notions. 

We now explain the main challenge in proving \Cref{thm_main}.
Starting with a $d$-dimensional pseudocharacter $f:R\map A$, there is a natural candidate for an associated determinant, namely
    \begin{equation}\label{eq_det}
    D_f(x)=\frac{1}{d!}\pse{f}{d}(x,\dots,x).
    \end{equation} 
It is easy to see that $D_f$ is an $A$-valued, degree $d$ homogeneous polynomial law.
The problem is to show that $D_f$ is multiplicative.
We prove the multiplicativity of $D_f$ in \Cref{cor_multiplicative}.
It directly follows from the following product formula, which is proven in Section $3$.
\begin{prop}\label{prop_n_n_identity}
Let $f:R\map A$ be a pseudocharacter of degree $d$.
Then
\[\pse{f}{d}(x_1,\dots,x_d)\pse{f}{d}(y_1,\dots,y_d)=\sum_{\sigma\in S_d}\pse{f}{d}(x_1y_{\sigma(1)},\dots,x_dy_{\sigma(d)}).\]
\end{prop} 
    
\Cref{prop_n_n_identity} is a special case of a more general product formula that we prove in \Cref{thm_main_identity}.


\section{A product formula}

\subsection{Combinatorial preparation}

Let $H$ be a semigroup, i.e.\ a set $H$ together with an associative binary operation.
We denote by $xy$ the product of $x,y\in H$.
Given $x_1,\dots,x_n\in H$, we denote by $\multi{x}=\multiset{x_1,\dots,x_n}\in H^n/S_n$ the multiset of these elements, here $S_n$ denotes the permutation group on $\braces{1,\dots,n}$.
The cardinally of the multiset $\multi{x}=\multiset{x_1,\dots,x_n}$ is denoted by $\car{\multi{x}}=n$.

By an \textit{$H$-multiset} we will always mean a finite multiset of the form $\multiset{x_1,\dots,x_n}$, for some $n$ and $x_1,\dots,x_n\in H$.

Let $\M(H)$ denote the free $\Z$-module on the set of all $H$-multisets, including the empty $H$-multiset, which we denote by $\emptyset_H$.
Thus, an element in $\M(H)$ is a formal sum
\[\sum_{i=1}^ma_i\multi{x}_i,\]
where $a_1,\dots,a_m\in\Z$ and $\multi{x}_1,\dots,\multi{x}_m$ are $H$-multisets.
We identify an $H$-multiset $\multi{x}$ with its image in $\M(H)$.
The $H$-multisets form a $\Z$-basis of $\M(H)$.
Next, we define a multiplication on $\M(H)$.

Let $[n]=\braces{1,\dots,n}$.
A \textit{partial bijection} from $[n]$ to $[m]$ is a triple $(I,J,\alpha)$, where $I\subset [n]$, $J\subset [m]$, and $\alpha:I\map J$ is a bijection.
We will use the notation $(I,J,\alpha):[n]\map [m]$.

Let $\multi{x}\in H^n/S_n$ and $\multi{y}\in H^m/S_m$ be $H$-multisets.
Choose an ordering of the elements of $\multi{x}$ and an ordering of the elements of $\multi{y}$:
\[\multi{x}=\multiset{x_1,\dots,x_n},\ \multi{y}=\multiset{y_1,\dots,y_m}.\]
Let $(I,J,\alpha):[n]\map [m]$ be a partial bijection.
We define an $H$-multiset $\multi{x}\atimes{\alpha}\multi{y}$ as follows.
\[\multi{x}\atimes{\alpha}\multi{y}=\multiset{x_s\ |\ s\in [n]\backslash I}\ \sqcup\ \multiset{y_t\ |\ t\in [m]\backslash J}\ \sqcup\ \multiset{x_iy_{\alpha(i)}\ |\ i\in I},\]
where $\sqcup$ is the disjoint union of multisets.
The cardinality of $\multi{x}\atimes{\alpha}\multi{y}$ is 
\[\car{\multi{x}\atimes{\alpha}\multi{y}}=\car{\multi{x}}+\car{\multi{y}}-\car{I}.\]
Note that we always have the empty partial bijection: $(\varnothing,\varnothing,\varnothing):[n]\map :[m]$, and
\[\multi{x}\atimes{\varnothing}\multi{y}=\multiset{x_1,\dots,x_n,y_1,\dots,y_m}.\]

\begin{defn}
Given $H$-multisets $\multi{x}=\multiset{x_1,\dots,x_n}$ and $\multi{y}=\multiset{y_1,\dots,y_m}$, we define their multiplication $\multi{x}\times \multi{y}$ in $\M(H)$ by
\[\multi{x}\times\multi{y}=\sum_{(I,J,\alpha):[n]\map [m]}\multi{x}\atimes{\alpha}\multi{y},\]
where the sum runs over all partial bijections $(I,J,\alpha):[n]\map [m]$.
\end{defn} 

Let $\multi{x}\in H^n/S_n$, $\multi{y}\in H^m/S_m$ and let $(I,J,\alpha):[n]\map [m]$ be a partial bijection.
Note that the $H$-multiset $\multi{x}\atimes{\alpha}\multi{y}$ is defined only when we have chosen specific orderings of the elements of $\multi{x}$ and of $\multi{y}$.
The $H$-multiset $\multi{x}\atimes{\alpha}\multi{y}$ depends on those choices of orderings.
However, $\multi{x}\times \multi{y}$ is independent of the orderings of $\multi{x}$ and  $\multi{y}$.

    \begin{exmp}\label{exmp_empty_is_unit}
    Let $\multi{x}\in H^n/S_n$
    Recall that $\emptyset_H$ is the empty $H$-multiset.
    Then $\multi{x}\times \emptyset_H=\emptyset_H\times \multi{x}=\multi{x}$.
    \end{exmp} 

\begin{exmp}
    Let $\multi{x}=\multiset{x_1,x_2}$, $\multi{y}=\multiset{y_1}$, and $\multi{z}=\multiset{z_1,z_2}$.
        \begin{align*}
        \multi{x}\times\multi{y}&=\multiset{x_1,x_2,y_1}+\multiset{x_1y_1,x_2}+\multiset{x_1,x_2y_1}, \text{ and}\\
        \multi{x}\times\multi{z}&=\multiset{x_1,x_2,z_1,z_2}+\multiset{x_1z_1,x_2,z_2}+\multiset{x_1,x_2z_1,z_2}\\
        &\hspace{1cm}+\multiset{x_1z_2,x_2,z_1}+ \multiset{x_1,x_2z_2,z_1}+ \multiset{x_1z_1,x_2z_2}+\multiset{x_1z_2,x_2z_1}.
        \end{align*} 
\end{exmp} 

\begin{exmp}\label{exm_free}
    Let $H$ be the free semigroup on the symbols $x_1,\dots,x_n,y_1,\dots,y_m$.
    Let $\multi{x}=\multiset{x_1,\dots,x_n}$ and $\multi{y}=\multiset{y_1,\dots,y_m}$.
    Then $\multi{x}\times \multi{y}$ is the sum of all of the $H$-multisets $\multiset{u_1,\dots,u_k}$ of the following form. 
    Each $u_i$ is equal to either $x_t$, or $y_s$, or $x_ty_s$ for some $t\in [n]$ and some $s\in [m]$. 
    Each $x_t$ appears (either alone or multiplied by some $y_s$) in $u_i$ for a unique $i\in [k]$.
    Similarly, each $y_s$ appears in $u_j$ for a unique $j\in [k]$.
    Moreover, any $H$-multiset that appears in the sum $\multi{x}\times \multi{y}$, appears with coefficient $1$.
    We note that $\max(n,m)\leq k\leq n+m$.
\end{exmp} 


We extend the multiplication $\times$ to $\M(H)$ $\Z$-linearly.
Thus, if $S_1=\sum_{i=1}^{m}a_i\multi{x}_i$ and $S_2=\sum_{i=1}^{n}b_i\multi{y}_i$, then
\[S_1\times S_2=\sum_{i=1}^m\sum_{j=1}^na_ib_j\multi{x}_i\times\multi{y}_j.\]
Since the set of $H$-multisets is a free basis of $\M(H)$ as a $\Z$-module, the extension of $\times$ to $\M(H)$ is well defined, and it exhibits both left and right distributivity.

Let $\psi:H_1\map H_2$ be a homomorphism of semigroups.
Applying $\psi$ on each entry of each $H_1$-multiset, we obtain a $\Z$-linear map $\M(\psi):\M(H_1)\map\M(H_2)$.
The next proposition is immediate from the definitions.

\begin{prop}\label{prop_functoriality}
For any $S,T\in \M(H_1)$, 
\[\M(\psi)(S\times T)=\M(\psi)(S)\times \M(\psi)(T).\]
\end{prop} 

The next result is the main combinatorial input.

\begin{prop}\label{prop_associativity}
Let $\multi{x},\multi{y}$, and $\multi{z}$ be $H$-multisets.
Then we have an equality of multisets
    \begin{equation}\label{eq_associativity}
    (\multi{x}\times \multi{y})\times \multi{z}=\multi{x}\times (\multi{y}\times \multi{z}).
    \end{equation} 
\end{prop} 

    \begin{proof}
    By \Cref{prop_functoriality}, it is enough to prove the claim for the free semigroup $H$ on the variables $x_1,\dots,x_n,y_1,\dots,y_m,z_1,\dots,z_k$.
    Let $\multi{x}=\multiset{x_1,\dots,x_n}$, $\multi{y}=\multiset{y_1,\dots,y_m}$, and $\multi{z}=\multiset{z_1,\dots,z_k}$ in $H$.
    Both sides of \Cref{eq_associativity} are equal to the sum of all of the $H$-multisets of the form $\multi{w}=\multiset{u_1,\dots,u_l}$, where each $w_i$ is either $x_s$, or $y_t$, or $z_w$, or $x_sy_t$, or $x_sz_w$, or $y_tz_w$, or $x_sy_tz_w$, for some $s\in [n],\ t\in [m],\ w\in [k]$.
    And for each $s\in [n],\ t\in [m],\ w\in [k]$, the variables $x_s$, $y_t$, $z_u$ appear exactly once in $\multi{w}$.
    Morevover, any such $\multi{w}$ appears with coefficient $1$ in the sum.
    \end{proof} 

\begin{cor}
Let $H$ be a semigroup.
Then $\M(H)$ together with the multiplication $\times$ is a unital ring.
\end{cor} 
    \begin{proof}
    The empty $H$-multiset $\emptyset_H$ is the multiplicative unit of $\M(H)$, as follows from \Cref{exmp_empty_is_unit}.
    We already know that $\times$ is left and right distributive.
    The associativity of $\times$ follows from distributivity together with \Cref{prop_associativity}.
    \end{proof}

\subsection{A product formula}

During this section, we take $H=(R,\cdot)$ to be the multiplicative semigroup of $R$, and we let $\M(R)$ stand for $\M((R,\cdot))$.
We fix a central functions $f:R\map A$.

Given $\multi{x}=\multiset{x_1,\dots,x_n}\in R^n/S_n$, we set
\[\pse{f}{n}(\multi{x})=\pse{f}{n}(x_1,\dots,x_n).\]
It is well known (see \cite{rouquier1996}) that if $f$ is central, then $\pse{f}{n}$ is a symmetric multilinear form.
Therefore, $\pse{f}{n}(\multi{x})$ is well defined.
By convention, $\pse{f}{}(\emptyset_H)=1$, where $\emptyset_H$ is the empty $H$-multiset.
We extend $\pse{f}{}$ to $\M(H)$, $\Z$-linearly.
If $S=\sum_{i=1}^ma_i\multi{x}_i\in  \M(R)$, then
\[\pse{f}{}(S)=\sum_{i=1}^ma_i\pse{f}{\car{\multi{x}_i}}(\multi{x}_i).\]

\begin{prop}[Reformulation of \Cref{eq_recursive}]\label{prop_product_formula_first_case}
Let $S\in \M(R)$ and $z\in R$.
Then
\[\pse{f}{}(S\times \multiset{z})=\pse{f}{}(S)\cdot f(z).\]
\end{prop}
    \begin{proof}
    By the linearity of $\pse{f}{}$ and the distributivity of the multiplication $\times$, it is enough to prove the proposition in the special case $S=\multi{x}=\multiset{x_1,\dots,x_n}$.
    In this case,
    \[\pse{f}{}(\multi{x}\times \multiset{z})=\pse{f}{n+1}(x_1,\dots,x_n,z_1)+\sum_{i=1}^n\pse{f}{n}(x_1,\dots,x_{i-1},x_iz_1,x_{i+1},\dots,x_n),\]
    and
    \[\pse{f}{n}(\multi{x})\cdot f(z)=f(z)\cdot\pse{f}{n}(x_1,\dots,x_n).\]
    The equality $\pse{f}{}(\multi{x}\times \multiset{z})=\pse{f}{n}(\multi{x})\cdot f(z)$ is now a direct consequence of \Cref{eq_recursive}.    
    \end{proof}

\begin{thm}\label{thm_main_identity}
Let $\multi{x}\in R^n/S_n$ and $\multi{y}\in R^m/S_m$.
Then
    \begin{equation}\label{eq_1}
    \pse{f}{}(\multi{x}\times \multi{y})
    =\pse{f}{n}(\multi{x})\cdot\pse{f}{m}(\multi{y})
    \end{equation} 
\end{thm} 
    \begin{proof}
    We fix $n$ and we fix $\multi{x}=\multiset{x_1,\dots,x_n}\in R^n/S_n$.
    The proof is by induction on the cardinality of $\multi{y}$.
    The case $\car{\multi{y}}=1$ was treated in \Cref{prop_product_formula_first_case}.

    Assume that \Cref{eq_1} holds whenever $\car{\multi{y}}<m$.
    Let $\multi{y}=\multiset{y_1,\dots,y_m}$, $\multi{y}'=\multiset{y_1,\dots,y_{m-1}}$, and $\multi{z}=\multiset{y_m}$.
    By \Cref{prop_associativity},
    \[(\multi{x}\times \multi{y}')\times \multi{z}=\multi{x}\times(\multi{y}'\times \multi{z})\]
    By the induction hypothesis, and since $\car{\multi{z}},\car{\multi{y}'}<m$, 
    \[\pse{f}{}\sog{(\multi{x}\times \multi{y}')\times \multi{z}}=\pse{f}{}(\multi{x}\times \multi{y}')\cdot \pse{f}{1}(\multi{z})=\pse{f}{n}(\multi{x})\pse{f}{m-1}(\multi{y}')f(y_m).\]
    Note that  
    \[\multi{y}'\times \multi{z}=
    \multi{y}+ [y_1y_m,y_2,\dots,y_{m-1}]+[y_1,y_2y_m,\dots,y_{m-1}]+\dots+[y_1,y_2,\dots,y_{m-1}y_m].\]
    Then
    \[\multi{x}\times (\multi{y}'\times \multi{z})=\multi{x}\times \multi{y}+\multi{x}\times[y_1y_m,y_2,\dots,y_{m-1}]+\dots+\multi{x}\times [y_1,y_2,\dots,y_{m-1}y_m],\]
    and therefore,
        \begin{align*}
        \pse{f}{}\sog{\multi{x}\times(\multi{y}'\times\multi{z})}
        &=\pse{f}{}(\multi{x}\times \multi{y})+\sum_{i=1}^{m-1}\pse{f}{}(\multi{x}\times [y_1,\dots,y_{i-1},y_iy_m,y_{i+1},\dots,y_{m-1}])\\
        &=\pse{f}{}(\multi{x}\times \multi{y})+\sum_{i=1}^{m-1}\pse{f}{}(\multi{x})\cdot\pse{f}{}([y_1,\dots,y_{i-1},y_iy_m,y_{i+1},\dots,y_{m-1}]).
        \end{align*} 
    The last equality follows from the induction hypothesis.
    Finally, 
        \begin{align*}
        \pse{f}{}(\multi{x}\times \multi{y})
        &=\pse{f}{}(\multi{x}\times(\multi{y}'\times \multi{z}))-\sum_{i=1}^{m-1}\pse{f}{}(\multi{x})\pse{f}{m-1}([y_1,\dots,y_{i-1},y_iy_m,y_{i+1},\dots,y_{m-1}])\\
        &=\pse{f}{}((\multi{x}\times \multi{y}')\times \multi{z})-\pse{f}{}(\multi{x})\sum_{i=1}^{m-1}\pse{f}{m-1}([y_1,\dots,y_{i-1},y_iy_m,y_{i+1},\dots,y_{m-1}])\\
        &=\pse{f}{}(\multi{x})\pse{f}{}(\multi{y}')f(y_m)-\pse{f}{}(\multi{x})\sum_{i=1}^{m-1}\pse{f}{m-1}([y_1,\dots,y_{i-1},y_iy_m,y_{i+1},\dots,y_{m-1}])\\
        &=\pse{f}{n}(\multi{x})\sog{\pse{f}{m-1}(y_1,\dots,y_{m-1})f(y_m)-\sum_{i=1}^{m-1}\pse{f}{m-1}(y_1,\dots,y_{i-1},y_iy_m,y_{i+1},\dots,y_{m-1})}\\
        &=\pse{f}{n}(\multi{x})\pse{f}{m}(\multi{y}).
        \end{align*} 
    The last equality follows from \Cref{eq_recursive}.
    \end{proof} 

\begin{cor}
The map $\pse{f}{}:\M(R)\map A$ is a ring homomorphism.
\end{cor} 
    \begin{proof}
    The only thing left to be shown is that $\pse{f}{}(S_1\times S_2)=\pse{f}{}(S_1)\cdot\pse{f}{}(S_2)$, for all $S_1,S_2\in \M(H)$.
    This follows from \Cref{thm_main_identity}, together with the distributivity of the multiplication $\times$, and the linearity of $\pse{f}{}$.
    \end{proof}

\section{A product formula for pseudocharacters and the proof of \Cref{thm_main}}

    \begin{proof}[Proof of \Cref{prop_n_n_identity}]
    Let $\multi{x}=\multiset{x_1,\dots,x_d}$ and $\multi{y}=\multiset{y_1,\dots,y_d}$.
    Let $(I,J,\alpha):[d]\map [d]$ be a partial bijection.
    If $\alpha$ is not a bijection, then $\pse{f}{k}(\multi{x}\atimes{\alpha}\multi{y})=0$, where $k=\car{\multi{x}\atimes{\alpha}\multi{y}}$.
    Indeed, we have $\car{I}<d$, so
    \[k=\car{\multi{x}\atimes{\alpha}\multi{y}}=\car{\multi{x}}+\car{\multi{y}}-\car{I}>d.\]
    Since $f$ is pseudocharacter of degree $d$, and $k>d$, we have $\pse{f}{k}(\multi{x}\atimes{\alpha}\multi{y})=0$.
    
    Now, if $(I,J,\alpha):[d]\map [d]$ is a bijection, then by definition,
    \[\multi{x}\atimes{\alpha} \multi{y}=\multiset{x_1y_{\alpha(1)},\dots,x_dy_{\alpha(d)}}.\]
    
    The proof now follows from \Cref{thm_main_identity} which asserts that $\pse{f}{d}(\multi{x})\pse{f}{d}(\multi{y})=\pse{f}{}(\multi{x}\times \multi{y})$.
    \end{proof}

\begin{cor}\label{cor_multiplicative}
Let $f:R\map A$ be a pseudocharacter of degree $d$.
Then the map $D_f$ defined in \Cref{eq_det} is multiplicative.
\end{cor} 
    \begin{proof}
    Let $x,y\in R$, and set $\multi{x}=\multiset{x,\dots,x}\in R^d/S_d$ and $\multi{y}=\multiset{y,\dots,y}\in R^d/S_d$.
    By \Cref{prop_n_n_identity},
        \begin{align*}
        D_f(x)D_f(y)
        &=\frac{1}{d!}\pse{f}{d}(x,\dots,x)\frac{1}{d!}\pse{f}{d}(y,\dots,y)
        =\frac{1}{d!^2}\sum_{\sigma\in S_d}\pse{f}{d}(xy,\dots,xy)\\
        &=\frac{1}{d!^2}d!\pse{f}{d}(xy,\dots,xy)
        =\frac{1}{d!}\pse{f}{d}(xy,\dots,xy)
        =D_f(xy).
        \end{align*} 
    \end{proof} 

    \begin{lem}\label{lem_simple_comp}
    Let $f:R\map A$ be a central function.
    Then
    \[\pse{f}{n}(x,1,\dots, 1)=f(x)\prod_{i=1}^{n-1}(f(1)-i).\]
    \end{lem} 
        \begin{proof}
        By induction on $n$.
        When $n=1$, both sides are equal to $f(x)$.
        Let $n>1$ and assume that the claim holds for $n-1$.
        By \Cref{eq_recursive},
            \begin{align*}
            \pse{f}{n}(x,1,\dots,1)
            &=f(1)\pse{f}{n-1}(x,1,\dots,1)-(n-1)\pse{f}{n-1}(x,1,\dots,1)\\
            &=(f(1)-(n-1))\pse{f}{n-1}(x,1,\dots,1)\\
            &=(f(1)-(n-1))f(x)\prod_{i=1}^{n-2}(f(1)-i)
            =f(x)\prod_{i=1}^{n-1}(f(1)-i).
            \end{align*} 
        \end{proof} 

    \begin{proof}[Proof of \Cref{thm_main}]
    Let $D$ be an $A$-valued $d$-dimensional determinant of $R$.
    Let $T:=\mathrm{Tr}_D$ be the trace function associated to $D$.
    By definition, $T(x)$ is the negative of the coefficient of $t^{d-1}$ in $D(t-x)$, i.e.
    \[D(t-x)=t^d-T(x)t^{d-1}+\dots.\]
    By $D_T$ we denote the polynomial law that assigns to any unital commutative $A$-algebra $B$, and to any $y\in B\otimes_A R$ the value
    \[D_T(y):=\frac{1}{d!}\pse{T^B}{d}(y,\dots,y).\]
    Here, $T^B$ stands for the function $T\otimes Id_B:R\otimes_A B\map B$.
    $D_T$ is clearly an $A$-valued, homogeneous of degree $d$, polynomial law.
    If $T$ is a $d$-dimensional pseudochracter, so is $T^B$ (\cite{Bellaiche_Chenevier_book}*{Section 1.2}), so by \Cref{cor_multiplicative}, $D_T$ is multiplicative.
    Therefore, $D_T$ is $d$-dimensional determinant over $A$.
    We claim that $D=D_T$.
    Since the map that takes a determinant to its trace is injective (\cite{chenevier2014p}*{Proposition 1.27}), it is enough to show that the trace associated to $D_T$ is $T$.
    Let $x\in R$.
    Since $\pse{T}{d}$ is a symmetric multilinear form, the coefficient of $t^{d-1}$ in the expression $D_T(t-x)=\frac{1}{d!}\pse{T}{d}(t-x,\dots,t-x)$ is
    \[\frac{1}{d!}\cdot d\cdot \pse{T}{d}(t,\dots,t,-x)
        =t^{d-1}\frac{1}{(d-1)!}\cdot \pse{T}{d}(-x,1,\dots,1).\]
    By \Cref{lem_simple_comp}, together with $T(1)=d$, we have that $\pse{T}{d}(-x,1,\dots,1)=(d-1)!T(-x)$.
    Then
    \[\frac{1}{d!}\cdot d\cdot \pse{T}{d}(t,\dots,t,-x)
    =t^{d-1}\frac{1}{(d-1)!}(d-1)!T(-x)=-t^{d-1}T(x).\]
    Therefore, the trace associated to $D_T$ is $T$.
    \end{proof}



\bibliographystyle{alpha}
\bibliography{main}
\end{document}